\documentclass{amsart}

\usepackage{latexsym,amssymb,graphicx,amscd,amsmath,amsthm}
\usepackage{amsmath,amssymb}

\newtheorem{thm}{Theorem}[section]
\newtheorem{lem}[thm]{Lemma}
\newtheorem{prop}[thm]{Proposition}
\newtheorem{cor}[thm]{Corollary}
\newtheorem{de}[thm]{Definition}
\newtheorem{rem}[thm]{Remark}

\newcommand{\BC}{{\mathbb{C}}}

\newcommand{\BZ}{{\mathbb{Z}}}

\newcommand{\CS}{{\mathcal{S}}}

\newcommand{\FD}{{\mathfrak{D}}}

\begin{document}

\title{A cellular basis for the generalized  Temperley-Lieb Algebra and Mahler Measure}
\bigskip

\author{Xuanting Cai}
\thanks{The author is with Mathematics Department, Louisiana State University, Baton Rouge, Louisiana  70803(email: xcai1@math.lsu.edu).}
\author{Robert G. Todd}
\thanks{The author is with Mathematics Department, University of Nebraska, Omaha, Nebraska,  68131(email: rtodd@unomaha.edu ).}

\date{}
\thispagestyle{empty}

\begin{abstract}

	Just as the Temperley-Lieb algebra is a good place to compute the Jones polynomial, the Kauffman bracket skein algebra of a disk with $2k$ colored points on the boundary, each with color $n$, is a good place to compute the $n^{th}$ colored Jones polynomial. Here, this colored skein algebra is shown to be a cellular algebra and a set of separating Jucys-Murphy elements is provided. This is done by explicitly providing the cellular basis and the JM-elements. Having done this several results of Mathas on such algebras are considered, including the construction of pairwise non-isomorphic irreducible submodules and their corresponding primitive idempotents. These idempotents are then used to define recursive elements of the colored skein algebra. Recursive elements are of particular interest as they have been used to relate geometric properties of link diagrams to the Mahler measure of the Jones polynomial. In particular, a single proof is given for the result of Champanerkar and Kofman, that the Mahler measure of the Jones and colored Jones polynomial converges under twisting on some number of strands.

\end{abstract}

\maketitle

\hspace{.2 in} {\bf Keywords}: {\em Colored Jones polynomial, Generalized Temperley-Lieb algebra, Graph basis, Mahler measure}

\bigskip
\thispagestyle{empty}
\setcounter{page}{1}

\section{Introduction}

 	The Temperley-Lieb algebra is equipped with a standard inner product  and is a convenient context for computing the Kauffman bracket, and thus the Jones polynomial of the closure of a given tangle. Here we define sub-algebras of the Temperley-Lieb algebra which are good places to compute colored Jones polynomials. Furthermore theses algebras are shown to be cellular by providing an explicit basis.  A family of separating Jucys-Murphy elements is also provided.  As a consequence of cellularity and the existence of the separating JM-elements, a complete set of pairwise non-isomorphic irreducible modules is provided along with the corresponding set of primitive idempotents. These idempotents are then used in to the understand Mahler measure of the Jones and colored Jones polynomials of recursively defined links, which include links which are formed by adding multiple twists on a collection of strands. 
 
The paper is organized as follows; Section \ref{GTLA} reviews the skein theory needed for the paper and introduce the $k^{th}$ $i-$colored Tempelry-Leib algebra $TL_{(k,i)}$. Section \ref{GB} gives a basis for $TL_{(k,i)}$.  Section \ref{cellularity} shows that the basis given in section \ref{GB} is a cellular basis. Section \ref{JM} and section \ref{CC} provides the JM elements and some consequences of cellularity. Lastly, section \ref{MMMT} states the results on Mahler measure.


\section{Generalization of Temperley-Lieb algebra $TL_n$}
\label{GTLA}
The $n^{th}$ Temperley-Lieb algebra $TL_n$ is a standard object of study when considering the Jones polynomial of knots and links. $TL_n$ is the algebra over $\mathbb{Z}[A^{\pm1}]$ generated by the elements $1, e_{1}, \ldots, e_{n-1}$ with relations $e_{i}^{2}=\delta e_{i},\ e_{i}e_{i\pm 1}e_{i}=e_{i}$ and $e_{i}e_{j}=e_{j}e_{i}$ when $|i-j|>1$, where $\delta=-A^{-2}-A^{2}$. The Temperley-Lieb algebra can be given a diagrammatic interpretation as the Kauffman bracket skein algebra of a disk $I \times [0,1]$ with $2n$ fixed points on the boundary ($n$ points fixed on $I \times \{0\}$ and $I \times \{1\}$ respectively). In this case the generators are given in figure \ref{generator}. Moreover $TL_{n}$ is equipped with a natural bilinear form. Let $L$ and $F$ be elements of $TL_{n}$. Then the inner product is defined by computing the Kauffman bracket of the link formed as in figure \ref{InnerProd}. This bilinear form was shown to be non-degenerate in \cite{C}. In what follows we will extend the ground ring to $\mathbb{Q}(A)$ so that the resulting algebra is semi-simple.

\begin{figure}[ht]
\[ 
1 =
\begin{array}{c}
\includegraphics[width=1in,height=1.2in]{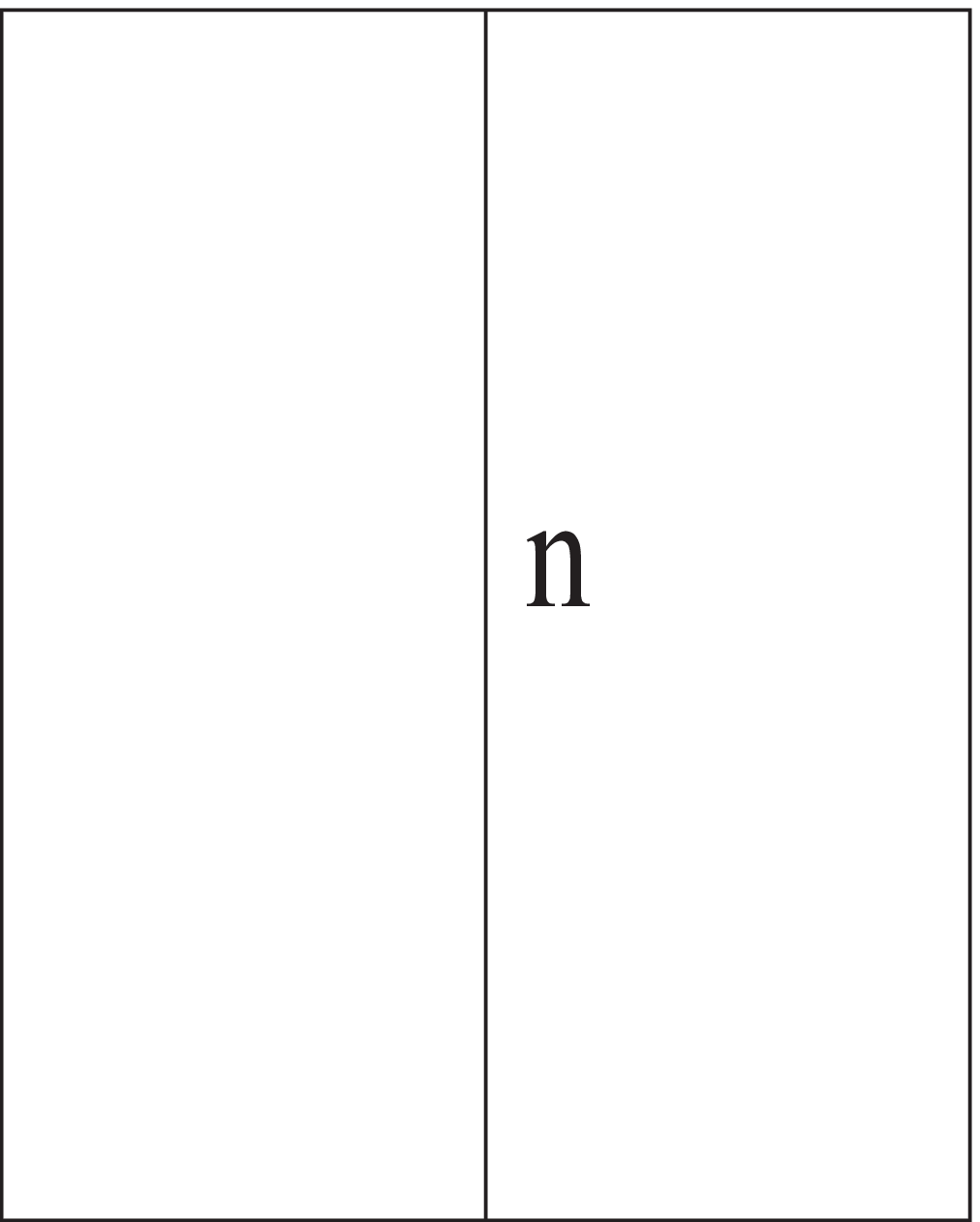}
\end{array}
e_i=
\begin{array}{c}
\includegraphics[width=1in,height=1.2in]{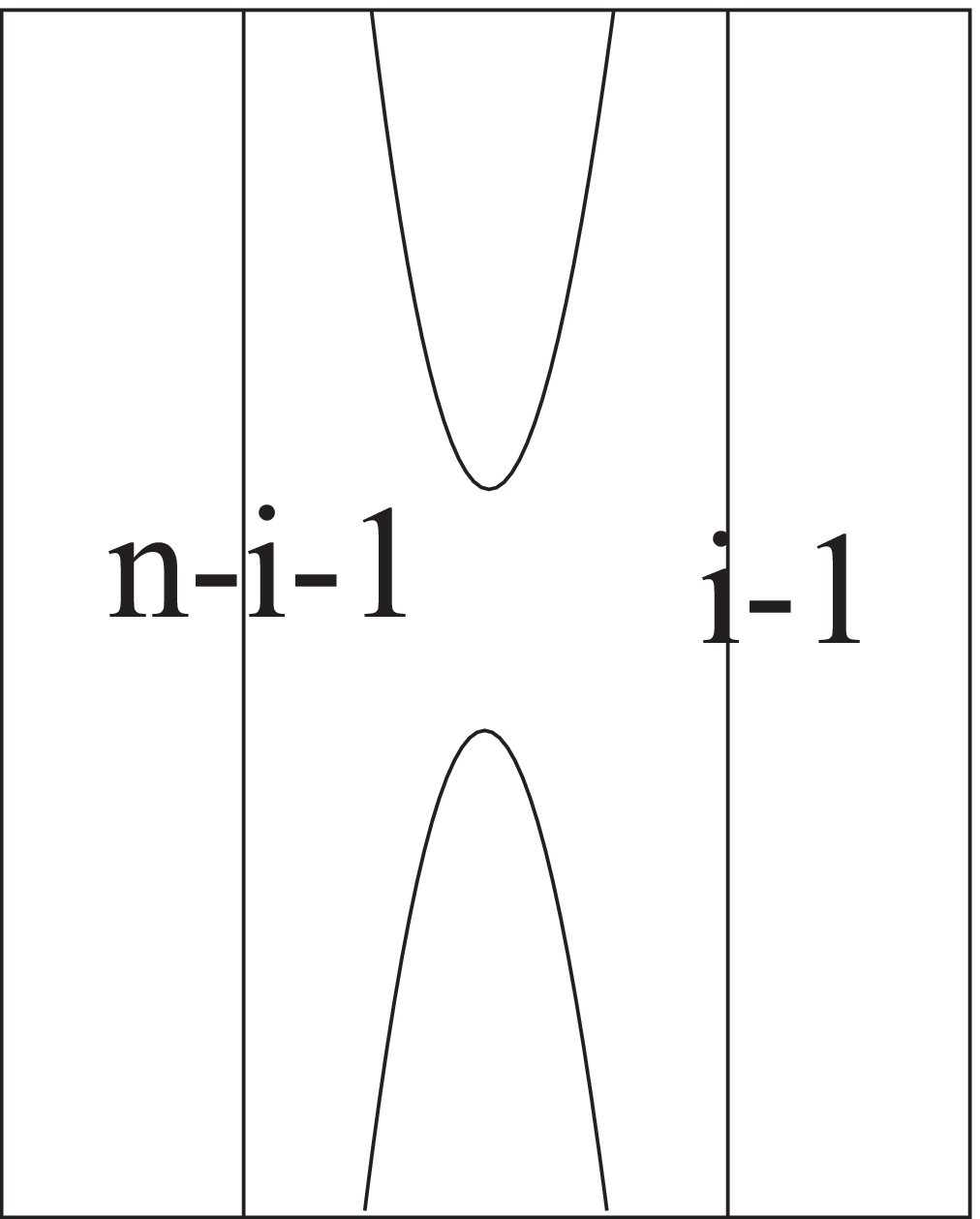}
\end{array}
\]
\caption{The integer $i$ beside the arc means $i$ parallel copies of the arc.}
\label{generator}
\end{figure}

\begin{figure}[ht]
\[\langle L,F\rangle=\langle\begin{array}{c}
 \includegraphics[scale=.75]{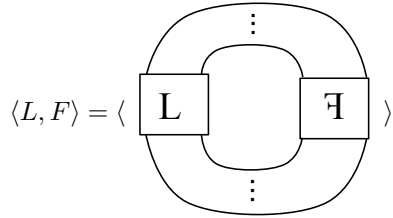} 
\end{array}\rangle
\]
\caption{The inner product via the Kauffman bracket}
\label{InnerProd}
\end{figure}

The inner product on $TL_{n}$ makes it a natural place to compute the Jones polynomial of a link. We want to find a natural place for computing the colored Jones polynomial of a link via an inner product. Recall that $TL_{n}$ contains a set of minimal idempotents $f_{1}, \ldots f_{n}$, constructed by Wenzl in \cite{Wenzl88}. $f_{n}$ is referred to as the $n^{th}$ Jones-Wenzl idempotent and provides a way to compute the $n^{th}$ colored Jones polynomial via the Kauffman bracket

\begin{de}
The $n^{th}$ colored Jones polynomial of link $L$ is defined to be
\begin{equation}
J_{n}(L)=A^{-(n^2+2n)w(L)}<L(n)>.
\notag
\end{equation}
Here $L(n)$ is the link $L$ with each component colored by $f_n$, $<>$ is the Kauffman bracket and $w(L)$ is the writhe of $L$.
\end{de}

Suppose that for particular link  $K$ with diagram $D$ there  are circles in the projection plane that divide the diagram into two $k$-tangles such that one may compute the Kauffman bracket of $D$ as $\langle D \rangle=\langle L, F\rangle$. To compute the $n^{th}$ colored Jones polynomial, color each component of D with the $n^{th}$ Jones-Wenzl idempotent. Coloring with an idempotent gives one the liberty to add an arbitrary number of idempotents to each component. Thus, color each component of the link $K$ with a copy of $f_{n}$ for each time it intersects the boundary of one of the circles. Then use the inner product in $TL_{kn}$ to compute the colored jones polynomial of the link $K$ as indicated in figure \ref{CKB}. 

\begin{figure}[ht]
\begin{center}
\includegraphics[scale=.5]{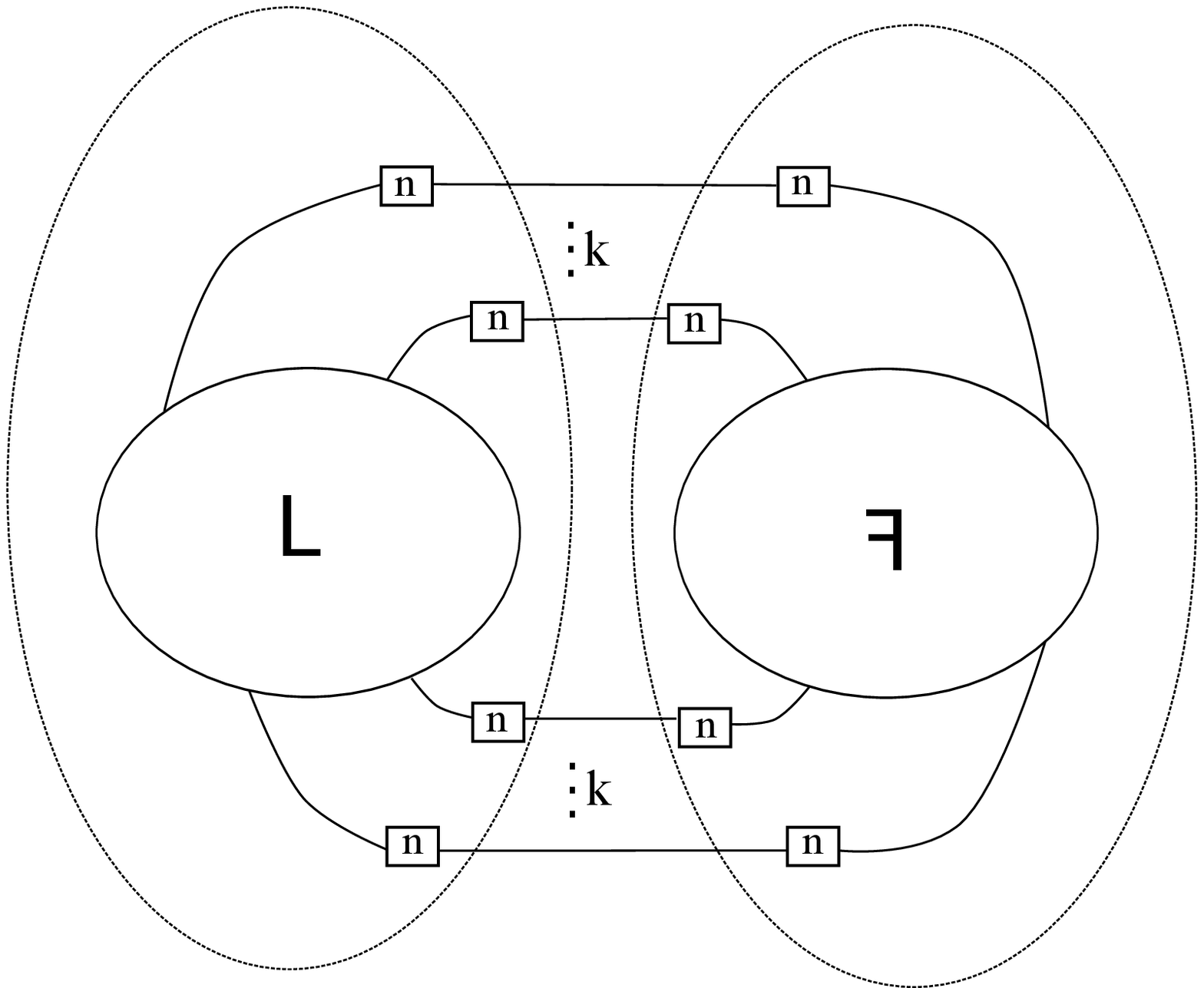}
\end{center}
\caption{}
\label{CKB}
\end{figure}

Motivated by the above observation consider the following map:

\begin{de}
Let $C_{i}:TL_{ik} \rightarrow TL_{ik}$ be the map where $C_{i}(T)$, for $T \in TL_{ik}$, is the element in figure \ref{embedding}, extended linearly in the obvious way.
\end{de}

\begin{thm}
For each $i$ $C_{i}$ is an algebra homomorphism.
\end{thm}
\begin{proof}
This statement follows from the idempotent property of $f_{n}$.
\end{proof}

This observation leads to the following definition.

\begin{de}
Let the $k^{th}$ $i$-colored Temperley-Lieb algebra $TL_{(k,i)}$ be the image of $TL_{ki}$ under the map $C_{i}$.
\end{de}

\begin{rem}
The $n$th Temperley-Lieb Algebra $TL_n$ is then identified as the $n^{th}$ $1$-colored Temperly-Lieb algebra.
\end{rem}

\begin{figure}[ht]
\begin{center}
\includegraphics[scale=.5]{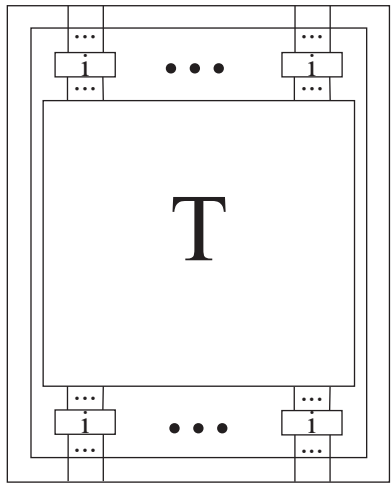}
\end{center}
\caption{Consider an element in $TL_{(k,i)}$ as an element in $TL_{ki}$.}
\label{embedding}
\end{figure}

  We can then equip $TL_{(k,i)}$ with the inner product that it inherits from $TL_{ki}$. This makes $TL_{(k,i)}$ a natural place to consider computations of the $i^{th}$ colored Jones polynomial of a link. 

\begin{rem}
The $k^{th}$ $i$-colored Temperley-Lieb algebra can also be viewed as the skein algebra of a disk with $2k$ points on the boundary each with color $i$. 
\end{rem}


\section{A Graph Basis Of $TL_{(k,i)}$}
\label{GB}

In \cite{C}, an orthogonal basis is constructed for the Temperley-Lieb algebra.
In this section, a similar basis is built for $TL_{(k,i)}$. In what follows, let an $i$ written next to a component denote it being colored with the $i^{th}$ Jones-Wenzl idempotent. The recoupling theory of such graphs is developed in \cite{kauffman94}. The reader is referred there for all but the most basic definitions, including the definitions of $\Delta_{n}$, $\theta(a,b,c)$ and $\lambda_{a}^{b,c}$.

\begin{de}
A triple of non-negative integers $(x,y,z)$ is called admissible if
\begin{enumerate}
\item $x+y+z$ is even;
\item $x\leq y+z$, $y\leq x+z$, $z\leq x+y$.
\end{enumerate}
\end{de}

 Given any admissible triple $(x,y,z)$ we may consider the following element in some Temperley-Lieb algebra.
 
 \begin{figure}[ht]
 \begin{center}
 \includegraphics[scale=.7] {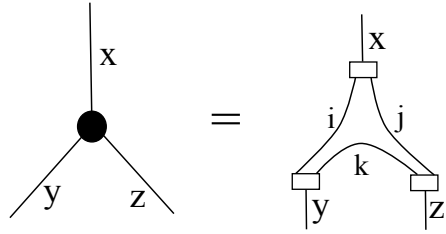}
 \end{center}
 \caption{An admissible colored trivalent graph}
 \label{ctg}
 \end{figure}

\begin{de}
Let $D_{a_1,...,a_{2n-1}}^i$ be the element of $TL_{(n,i)}$ in  figure \ref{D},
where $a_i$ satisfies:
\begin{enumerate}
\item $a_1=a_{2n-1}=i$;
\item $a_k\in \mathbb{N}$ for all $k$;
\item $(a_k,a_{k-1},i)$ is admissible for all $k$.
\end{enumerate}
Let $\FD(n,i)$ be the collection of all these elements.
\end{de}

\begin{figure}[ht]
\centering
\includegraphics{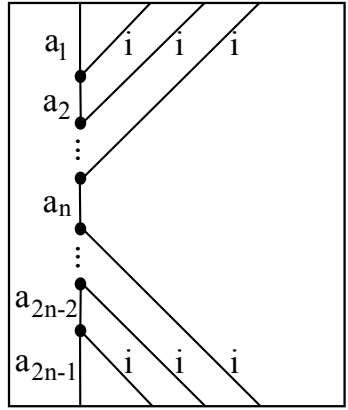}
\caption{At each black dot, the graph is admissible.}
\label{D}
\end{figure}

\begin{lem}
\label{norm}
Suppose $(a_1,...,a_{2n-1})$ and $(b_1,...,b_{2n-1})$ satisfy all conditions above,
then
\begin{eqnarray*}
&&<D_{a_1,...,a_{2n-1}}^i,D_{b_1,...,b_{2n-1}}^i>\notag \\
&=&\Delta_{a_{2n-1}}\prod_{j=1}^{2n-1}{\delta_{a_jb_j}}\prod_{j=1}^{2n-2}{\frac{\theta(a_{j+1},a_j,i)}{\Delta_{a_{j+1}}}}.
\end{eqnarray*}
Moreover, it is easy to see that it is nonzero.
\end{lem}
\begin{proof}
When $n=2$, by \cite[Lemma 4.5]{C} and direct computation,
\begin{eqnarray*}
&&<D_{a_1,a_2,a_3}^i,D_{b_1,b_2,b_3}^i>\notag\\
&=&\Delta_{a_3}\delta_{a_1b_1}\delta_{a_2b_2}\delta_{a_3b_3}\frac{\theta(a_2,a_1,i)}{\Delta_{a_2}}\frac{\theta(a_3,a_2,i)}{\Delta_{a_3}}.
\end{eqnarray*}
Thus the formula is true for $n=2$.
Now suppose the formula is true when $n=k-1$. Consider $n=k$. Then
\begin{eqnarray*}
&&< D_{a_1,...,a_{2n-1}}^i,D_{b_1,...,b_{2n-1}}^i > \notag \\
&=&\delta_{a_1b_1}\delta_{a_{2n-1}b_{2n-1}}\frac{\theta(a_2,a_1,i)}{\Delta_{a_2}}\frac{\theta(a_{2n-1},a_{2n-2},i)}{\Delta_{a_{2n-2}}}<D_{a_2,...,a_{2n-2}}^i,D_{b_2,...,b_{2n-2}}^i> \notag \\
&=&\delta_{a_1b_1}\delta_{a_{2n-1}b_{2n-1}}\frac{\theta(a_2,a_1,i)}{\Delta_{a_2}}\cdots\frac{\theta(a_{2n-1},a_{2n-2},i)}{\Delta_{a_{2n-2}}}\times \notag \\ &&\delta_{a_2b_2}...\delta_{a_{2n-2}b_{2n-2}}\frac{\theta(a_3,a_2,i)}{\Delta_{a_3}}\frac{\theta(a_{2n-2},a_{2n-3},i)}{\Delta_{a_{2n-2}}}\Delta_{a_{2n-2}} \notag \\
&=&\Delta_{a_{2n-1}}\delta_{a_1b_1}...\delta_{a_{2n-1}b_{2n-1}}\frac{\theta(a_2,a_1,i)}{\Delta_{a_2}}\cdots\frac{\theta(a_{2n-1},a_{2n-2},i)}{\Delta_{a_{2n-1}}}.
\end{eqnarray*}
Thus the formula holds for $n=k$.
Hence, by induction, the formula holds.
\end{proof}

\begin{cor}
\label{independent}
$D_{a_1,\cdots,a_{2n-1}}^i$'s are linearly independent.
\end{cor}
\begin{proof}
This follows directly from Lemma \ref{norm}.
\end{proof}

\begin{lem}
\label{identity}
The identity in $TL_{(k,i)}$ can be written as linear combination of $D_{a_1,\cdots,a_{2n-1}}^i$'s.
\end{lem}
\begin{proof}
We prove the result by induction on k.
Consider the standard fusion identity from \cite{kauffman94} as given in figure \ref{fusion}.
\begin{figure}[ht]
\[
\begin{array}{c}
\includegraphics[scale=.7]{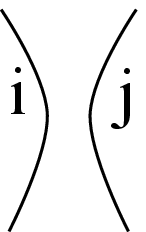}
\end{array}
= \sum_{k\ admissible} \frac{\Delta_{k}}{\theta(i,j,k)}
\begin{array}{c}
\includegraphics[scale=.8]{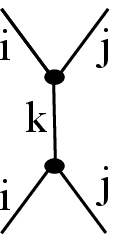}
\end{array}
\]
\caption{Fusion.}
\label{fusion}
\end{figure}

Notice that when $i=j$ one has the identity for $TL_{(2,i)}$. Thus the result holds when $k=2$.
Now suppose the result is true for $k=n$. Consider the case when $k=n+1$.
Just as $TL_{n}$ includes naturally into $TL_{n+1}$, $TL_{(n,i)}$ includes into $TL_{(n+1,i)}$. Thus apply the inductive hypothesis to the first $n$ strands of the identity element in $TL_{(n+1,i)}$, as indicated in the left hand side of \ref{induction}. Performing the fusion operation in each term on the right hand side of figure \ref{induction} leads to the conclusion.

 \begin{figure}[ht]
\[
\begin{array}{c}
\includegraphics{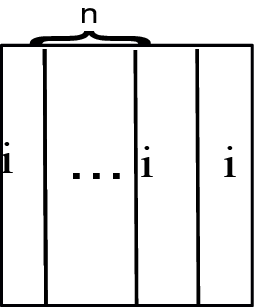}\hskip 0.1in
\end{array}
 =\sum_{ \bar{a}=(a_{1},\ldots, a_{2n-1})} C_{\bar{a}}
\begin{array}{c}
\includegraphics{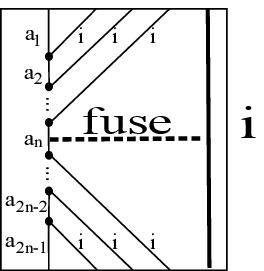}\hskip 0.1in
\end{array}
\]
\caption{Induction step.}
\label{induction}
\end{figure}
\end{proof}

\begin{lem}
\label{span}
Any element in $TL_{(n,i)}$  can be written as linear combination of $D_{a_1,\cdots,a_{2n-1}}^i$'s.
\end{lem}
\begin{proof}
Consider an arbitrary element $T \in TL_{(n,i)}$. 
In neighborhoods of the $n$ colored strands on either side of $T$ (in the dotted boxes) apply lemma \ref{identity}, as indicated on the left of figure \ref{DoubleFusion}.
 Now consider the dotted middle box on the right diagram in figure \ref{DoubleFusion}.
Consider it as an element in the skein module $\CS(I\times I,2)$ with $2$ colored points on the boundary.
It is well known that this skein module is either $0$ dimensional or $1$ dimensional, being 1 dimensional when the two points have the same color. In this case a single component labeled with the corresponding idempotent is the basis.
Therefore, one can write the element in the middle box to be a non-zero constant times a single component labeled with the idempotent $a_{n}$, when $a_{n}=b_{n}$, and zero otherwise.
Thus each element can be written as linear combination of $D_{a_1,\cdots,a_{2n-1}}$.

\begin{figure}[ht]
\[
\begin{array}{c}
\includegraphics{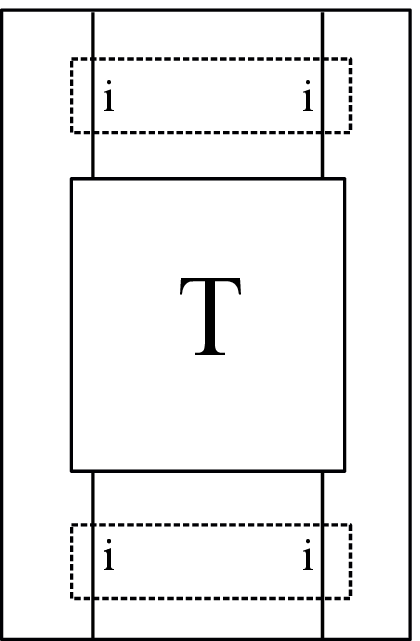}
\end{array}
= \sum_{\bar{a},\ \bar{b}} C_{\bar{a}, \bar{b}}
\begin{array}{c}
\includegraphics{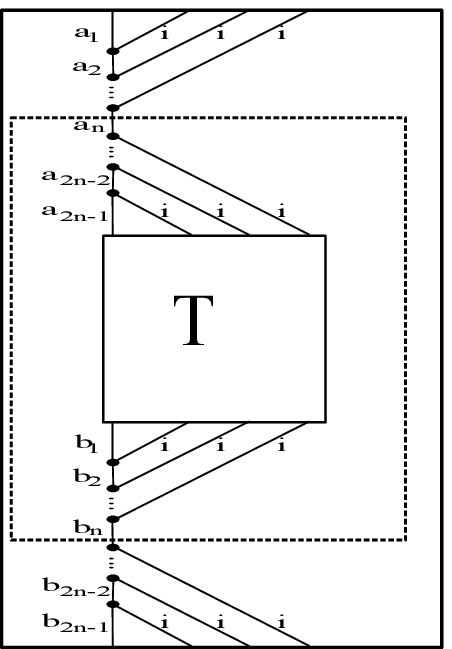}
\end{array}
\]
\caption{Linear combination.}
\label{DoubleFusion}
\end{figure}

\end{proof}

By Corollary \ref{independent} and Lemma \ref{span}, the collection $\{D_{a_1,\cdots,a_{2n-1}}^i\}$ is linearly independent and spans $TL_{(n,i)}$. 

\begin{thm}
$\{D_{a_1,\cdots,a_{2n-1}}^i\}$ is a basis of $TL_{(n,i)}$.
\end{thm}

\begin{rem}
P. Gilmer has also proved that $\{D_{a_1,\cdots,a_{2n-1}}^i\}$ is a basis in \cite{G}. There he used the skein module interpretation and decomposed any skein element in $TL_{(n,i)}$ into unions of $3$-balls with $3$ colored points on the boundary.
\end{rem}


\section{Cellularity of the graph basis}
\label{cellularity}

Having constructed a basis for the $k^{th}$ $i$-colored Temperley-Lieb algebra we will now show that this basis is cellular as defined in \cite{mathas2008}. 

\begin{de} A \textit{cell datum} for an unital R-algebra algebra $A$ is a triple $ (\Lambda, T, C)$ where $\Lambda= ( \Lambda, >)$ is a finite poset, $T(\lambda)$ is a finite set for each $\lambda \in \Lambda$ and
$$ C: \coprod_{\lambda \in \Lambda} T(\lambda) \times T(\lambda) \rightarrow A; (s,t) \mapsto a^{\lambda}_{st}$$

is an invective map of sets such that: 
\begin{itemize}
\item $\{ a_{st}^{\lambda} | \lambda \in \Lambda; s,t \in T(\lambda)\}$ is an R-free basis of A
\item For any $x \in A$ and $t \in T( \lambda)$ there exist scalars $r_{tvx} \in R$ such that for any $s \in T(\lambda),$
$$a^{\lambda}_{st}x=\sum_{v \in T(\lambda)} r_{tvx} a_{sv}^{\lambda} (mod A^{\lambda})$$

where $A^{\lambda}$ is the $R$-submodule of A with basis 
$\{a^{\mu}_{yz}| \mu > \lambda\ and\ y,z \in T(\mu)\}$

\item The $R$-linear map determined by 
$\ast:A\rightarrow A; a^{\lambda}_{st}=a^{\lambda}_{ts}$, for all $\lambda \in \Lambda$ and $s,t \in T(\lambda)$, is an anti-isomorphism of A.
\end{itemize}
\end{de} 

\begin{de}
A sequence $s=(s_{1},\ldots,s_{k})$ $i$-admissible if
\begin{enumerate}
\item $s_{1}=i$ 
\item $(s_{j-1},s_{j},i)$ is an admissible triple for $1<j\leq k$
\end{enumerate}

Furthermore we say the weight of $s$, $\omega(s)$, is defined by $\omega(s)=s_{k}$.
\end{de}

\begin{lem}
\label{weights}
	For $s=(s_{1},\ldots,s_{k})$, an $i$-admissible sequence of length $k$, 
	$\omega(s) \in \{ki-2j | j=0,\ldots, \lfloor \frac{ki}{2} \rfloor\}$ and furthermore, all such weights are achieved by an admissible sequence.
\end{lem}

\begin{proof}
Clearly the statement holds when $k=1$. Suppose now the statement holds for all $i$-admissible sequences of length k. Suppose that $s=(s_{1},\ldots,s_{k+1})$ is $i$-admissible. In particular  $(s_{k},s_{k+1},i)$ is an admissible triple. By our inductive hypothesis $s_{k} \in \{ki-2j | j=0,\ldots, \lfloor \frac{ki}{2} \rfloor\}$. Applying the definition of an admissible triple we find that $$0 \leq s_{k+1} \leq (k+1)i-2j$$ and that $(k+1)i-2j+s_{k+1}$ is even for each $ j=0\ldots \lfloor \frac{ki}{2}\rfloor$. In particular we see that $(s_{k},s_{k+1},i)$ is admissible if and only if

\begin{enumerate}
\item $0 \leq s_{k+1} \leq (k+1)i$ and
\item $s_{k+1}$ and $(k+1)i$ have the same parity.
\end{enumerate}

Thus $s_{k+1} \in \{(k+1)i-2j | j=0,\ldots, \lfloor \frac{(k+1)i}{2}\rfloor\}$

\end{proof}

\begin{de}
	For an $i$-admissible sequence $s=(s_{1}, \ldots, s_{k})$ define
\begin{equation}
\eta(s)=\prod_{j=1}^{k-1}\frac{\theta(s_{j+1},s_j,i)}{\Delta_{s_{j+1}}}.
\notag
\end{equation}
\end{de}

It is clear that for an $i$-admissible sequence $s$, $\eta(s)$ is non-zero.

\begin{de}
\label{CDatum1}
	Let 
	$\Lambda_{k,i}=\{ki-2j | j=0,\ldots, \lfloor \frac{ki}{2} \rfloor\}$ with the natural order. For $\lambda \in \Lambda_{k,i}$ let $T(\lambda)$ be the collection of all $i$-admissible sequences of length $k$ with weight $\lambda$. Consider the following map
	$$ C : \prod_{ \lambda \in \Lambda_{k,i}} T(\lambda) \times T(\lambda) \rightarrow TL_{(k,i)}$$

as follows; say $ s=(s_{1}, \ldots,s_{k}=\lambda)$ and $t=(t_{1},\ldots,t_{k}=\lambda)$
$$(s,t) \mapsto  \frac{1}{\eta(s)}* D^{i}_{s_{1}, \ldots,s_{k}=\lambda=t_{k},t_{k-1}, \ldots, t_{1}}$$

\

\end{de}
	 Letting $G_{s,t}$ denote $C(s,t)$ (understanding that $s$ and $t$ must be $i$-admissible sequences of the same length and weight),  $\{G_{s,t}| s, t \in T(\lambda), \lambda \in \Lambda_{i,k}\}$ is a  basis for the $k^{th}$ $i$-colored Temperley-Lieb algebra $TL_{(k,i)}$. 
	 
\begin{lem}
\label{GBProd}
$$G_{s,t}\cdot G_{u,v}=
 \begin{cases}
 G_{s,v} & \text{if}\ t=u \\
 0 & \text{otherwise}
 \end{cases}$$
 \end{lem}
	
\begin{proof}

\begin{figure}[!ht]
\[G_{s,t}\cdot G_{u,v}=
\frac{1}{\eta(s)} 
\begin{array}{c} \includegraphics[scale=.5]{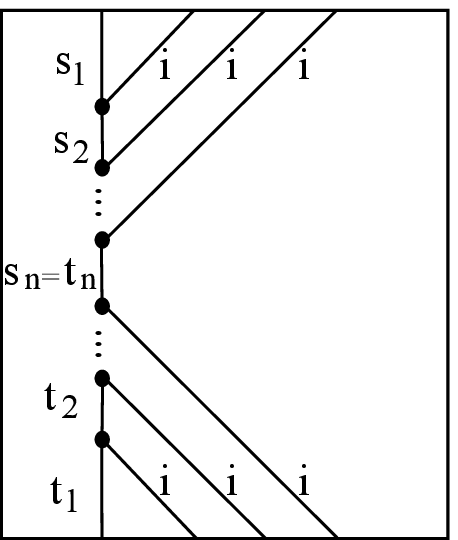}
\end{array}
\cdot \frac{1}{\eta(u)}
\begin{array}{c}
\includegraphics[scale=.5]{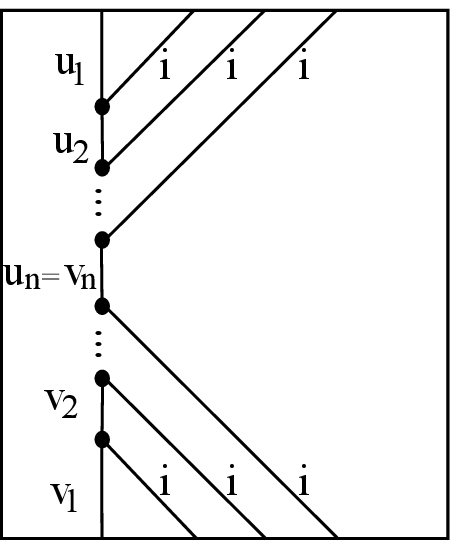}
\end{array} 
=\frac{1}{\eta(s)} \frac{1}{\eta(u)}
\begin{array}{c}
\includegraphics[scale=.5]{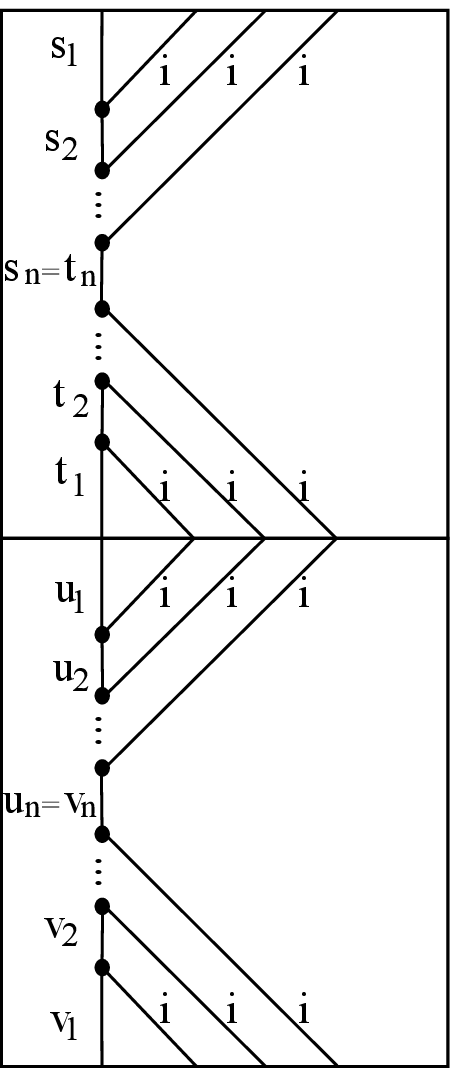}
\end{array}
\]

\caption{The product of two elements of the graph basis}
\label{GBProdComp}
\end{figure}
Let $s=(s_{1},\ldots s_{n}),\ t=(t_{1}, \ldots, t_{n}),\ u=(u_{1}, \ldots, u_{n}),$ and $v=(v_{1}, \ldots, v_{n})$.
After repeatedly applying the bubble relation from \cite{kauffman94} to the right-most diagram in figure \ref{GBProdComp} the result follows.

\end{proof}

\begin{prop}
\label{CDatum2}
Say that $ x \in TL_{(k,i)}$ and $t \in T(\lambda^{t})$. Then there are scalars $r_{tvx} \in \mathbb{Q}(A)$ such that for any $s \in T(\lambda^{t})$
$$G_{s,t}x=\sum_{v \in T(\lambda^{t})} r_{tvx}G_{s,v}$$
\end{prop}

\begin{proof}
Write $x$ as a linear combination of the graph basis elements,
$$x=\sum_{\lambda \in \Lambda_{k,i}}(\sum_{u,v \in T(\lambda)} p^{\lambda}_{u,v}G_{u,v})$$
where each $p_{u,v}^{\lambda} \in \mathbb{Q}(A)$. By Lemma \ref{GBProd} it follows that
$$ G_{s,t} \cdot x= \sum_{v \in T(\lambda^{t})} p^{\lambda^{t}}_{t,v} G_{s,t}\cdot G_{t,v}=\sum_{v \in T(\lambda^{t})} p^{\lambda^{t}}_{t,v} G_{s,v}$$
The lemma follows letting $r_{tvx}=p^{\lambda^{t}}_{t,v}$.
\end{proof}



\begin{prop}
\label{CDatum3}
	Consider the $\mathbb{Q}(A)$-linear map determined by $*:TL_{(k,i)} \rightarrow TL_{(k,i)}$ defined by $(G_{s,t})^{*}=G_{t,s}$. This map is an anti-isomorphism of $TL_{(k,i)}$
\end{prop}

\begin{proof}
It suffices to consider only the case of a non-zero product of basis elements. Thus consider 
$$(G_{s,t} \cdot G_{t,v})^{*}=
(G_{s,v})^{*}=G_{v,s}=G_{v,t}\cdot G_{t,s}=(G_{t,v})^{*} \cdot (G_{s,t})^{*}$$

Thus $*$ is an anti-homomorphism. It follows that $*$ is an isomorphism as it is a bijection from the set of basis elements to itself.
\end{proof}

\begin{rem}
	Note that this anti-isomorphism is the extension of the standard anti-isomorphism in the Temperley-Lieb algebra given by reflecting an element in a horizontal line. 
\end{rem}

\begin{thm}
	The $k^{th}$ $i$-colored Temperley-Lieb algebra $TL_{(k,i)}$ is a cellular algebra with cell datum $(\Lambda_{k,i}, T, C)$
\label{CA}
\end{thm}
\begin{proof}
	This follows directly from definition \ref{CDatum1}, proposition \ref{CDatum2} and proposition \ref{CDatum3}
	\end{proof}

	A Young diagram is a left justified array of boxes $d:=[d_{1},\ldots, d_{k}]$ with $d_{i}$ boxes in the $i^{th}$ row and with $d_{i} \geq d_{i+1}$. Here we will be concerned with 2-row Young diagrams. For a 2-row Young diagram $d=[d_{1},d_{2}]$ we define the weight of the Young diagram to be $\omega(d)=d_{1}-d_{2}$. When considering the basis for the standard Temperley-Lieb algebra constructed by Wenzl one considers the Bratteli diagram of 2-row Young diagrams. The Bratteli diagram along with the weights of each diagram are depicted in figure \ref{bratelli}.

	 Notice that the collection of weights of 2-row Young diagrams with $k*i$ boxes is $\Lambda_{k,i}$. It is straight forward to see that in the case of the standard Temperley-Lieb algebra $TL_{k}$, i.e., the $k^{th}$ 1-colored Temperley-Lieb algebra, the collection of admissible sequences of weight $\lambda \in  \Lambda_{k,1}$ correspond precisely to paths in the Brattelli diagram that start at the top and end at the Young diagram of appropriate weight. Thus, similar to \cite{Wenzl88} the graph basis can be indexed by pairs of standard Young tableaux. For $TL_{(k,i)}$ a similar diagram can be drawn where the graph basis is indexed by pairs of paths. In particular, consider the diagram where the first row is the Young diagram $[i,0]$ and the $k^{th}$ row consists of all 2-row Young diagrams with $ki$ boxes, increasing by weight from left to right. Say that $d_{1}$ is a diagram in row $k-1$ and $d_{2}$ is a diagram in row $k$. There is an edge in the diagram connecting $d_{1}$ to $d_{2}$ exactly when $(\omega(d_{1}),\omega(d_{2}),i)$ is an admissible triple. Then the graph basis for $TL_{(k,i)}$ is indexed by pairs of paths in the diagram. The branching diagram for $TL_{(3,2)}$ is in figure \ref{BW2}.
	
\begin{figure}[ht]
\begin{center}
\includegraphics[scale=.7]{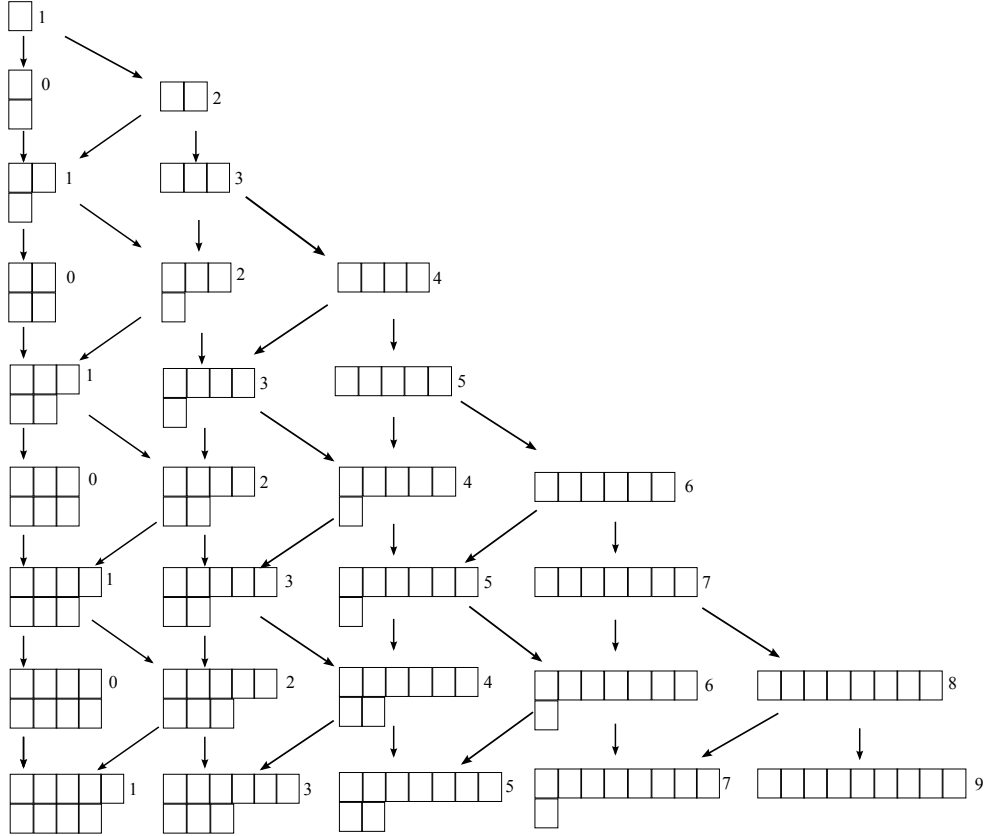}
\end{center}
\caption{Bratteli diagram}
\label{bratelli}
\end{figure}

\begin{figure}[ht]
\begin{center}
\includegraphics[scale=.8]{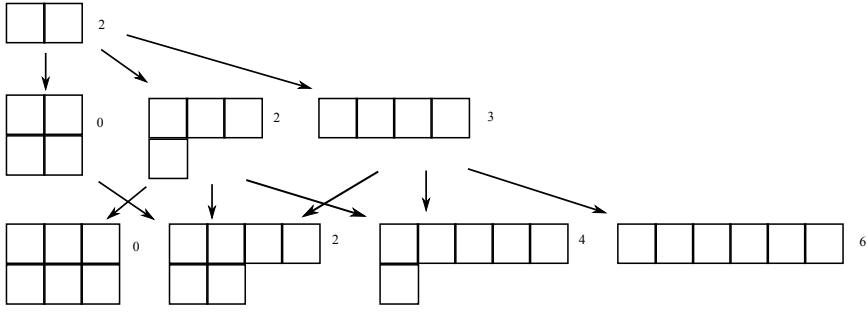}
\end{center}
\caption{Branching diagram for $TL_{(3,2)}$}
\label{BW2}
\end{figure}

\section{ A seperating set of JM-elements}
\label{JM}
	The Jucys-Murphy elements were first used in the representation theory of the symmetric group \cite{J1,murphy81}. A history of the seminormal representation theory of Weyl groups and the role of the JM-elements can be found in 
\cite{ram97}. JM-elements have been found for many different cellular algebras including the Hecke algebras (of various types) and the affine Temperley-Lieb algebra (see \cite{halverson09} for example). Furthermore, the JM-elements have been axiomatized for algebras built with Jones' basic construction (these include Brauer algebras, BMW algebras, and partition algebras) \cite{goodman11,goodman11a}.
	
\begin{de} \cite{mathas2008}
		Let $A$ be a cellular algebra with cell datum $( \Lambda, T, C)$. A family of JM-elements $\{L_{1}, \ldots, L_{M}\}$ is a family of mutually commuting elements along with scalars $\{c_{t}(i)| t \in T(\Lambda) 1\leq i\leq M\}$. Such that $(L_{i})^{*}=L_{i}$ and for all $\lambda \in \Lambda$ and $s,t \in T(\lambda)$ 

\begin{equation}
L_{i} \cdot a_{s,t}^{\lambda} =c_{s}(i) a_{s,t}^{\lambda} + \sum _{ u \rhd s} r_{s,u} a_{u,t}^{\lambda}\ mod (A^{\lambda})
\end{equation}
 Where $(A^{\lambda})$ is the submodule generated by $\{a^{\mu}_{s,t}|  \mu > \lambda\}$. Futhermore, we say that the JM-elements are separating if for $s \rhd t \in T(\Lambda)$  then for some i, $c_{s}(i) \neq c_{t}(t)$.
		
\end{de}
	The following is a typical definition of the JM-elements for the Temperley-Lieb algebra (see \cite{halverson09}). Here we generalize to the $k^{th}$ $i$-colored Temperley-Lieb algebra $TL_{(k,i)}$.
\begin{de}
 Let $L_{1}$ be the identity element in $TL_{(k,i)}$, and let $L_{j}$ be given as in figure \ref{JMElements}.
 
 \begin{figure}[ht]

 \[ L_{j} = 
 \begin{array}{c}
 \includegraphics{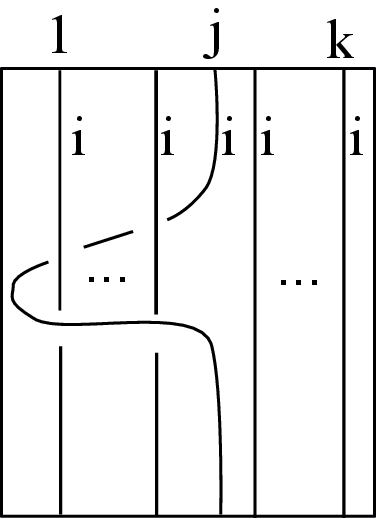}
 \end{array}
 \]
  \label{JMj}
  \caption{JM-elements}
   \label{JMElements}
  \end{figure}
 \end{de}

\begin{lem}
\label{JMComponents}
	Consider $G_{s,t} \in TL_{(k,i)}$ where $s=(s_{1}, \ldots, s_{k})$ and $t=(t_{1}, \ldots, t_{k})$ (assuming $s_{k}=t_{k}$ of course). Then
	
$$ < L_{1},G_{s,t} >=
\begin{cases}
\Delta_{s_{k}} & \text{if}\ s=t \\
0 & \text{otherwise}
\end{cases}
$$
and 
$$<L_{j},G_{s,t}>=
\begin{cases}
(\lambda_{s_{j}}^{s_{j-1},i})^{2} \Delta_{s_{k}} & \text{if}\ s=t \\
0 & \text{otherwise}
\end{cases}
$$
\end{lem}

\begin{proof}

In figure \ref{L1IP}, after repeated application of the bubble identity, the first result follows. In figure \ref{LjIP} one applies the bubble identity $j-1$ times, then the twist identity twice and then the bubble identity several more times to confirm the second result.

\begin{figure}[!ht]
\[ < L_{1},G_{s,t} >=\frac{1}{\eta(s)}<
\begin{array}{c}
\includegraphics[scale=.8]{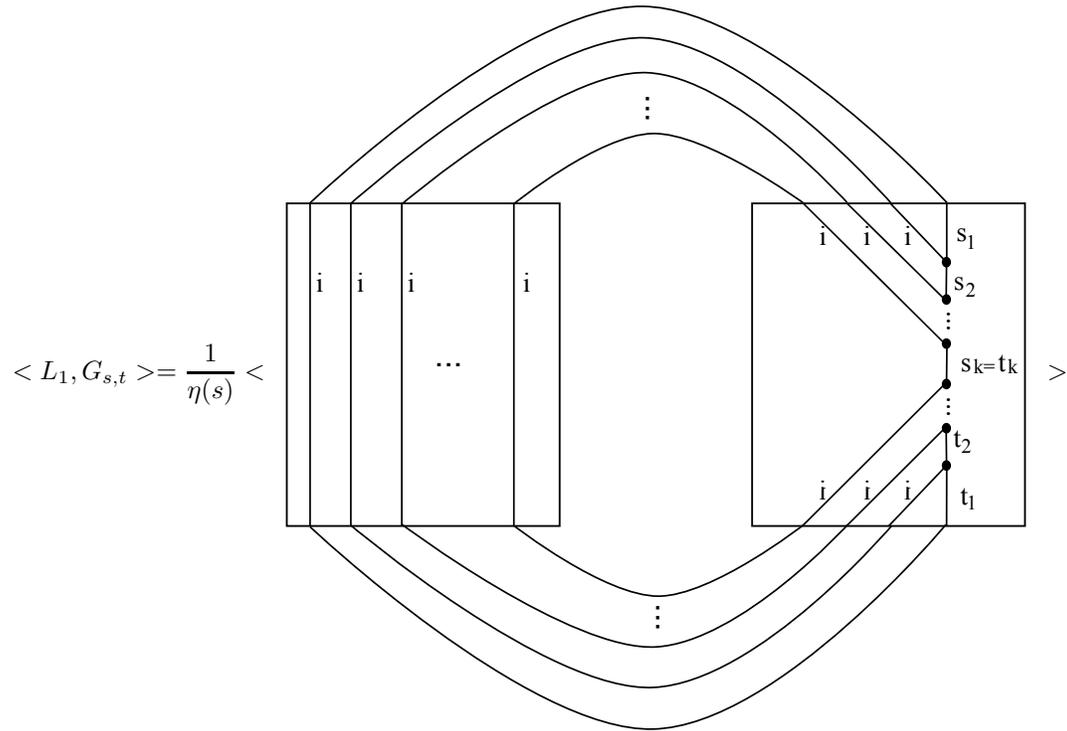}
\end{array}>
\]
\caption{Inner product with the identity element which is $L_{1}$}
\label{L1IP}
\end{figure}

\begin{figure}[!ht]
\[ < L_{j},G_{s,t} >=\frac{1}{\eta(s)}<
\begin{array}{c}
\includegraphics[scale=.8]{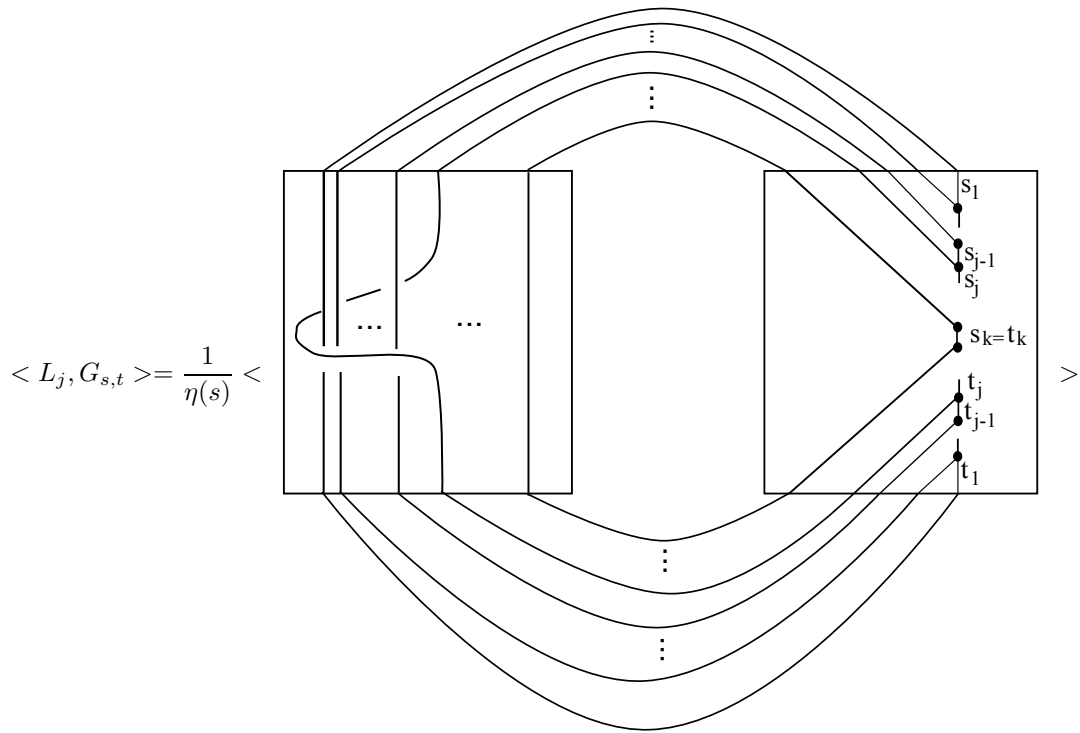}
\end{array}>
\]
\caption{Inner Product with $L_{j}$}
\label{LjIP}
\end{figure}

\end{proof}
	Similar computations as in the proof of lemma \ref{JMComponents} lead to the following statement
\begin{lem}
\label{GBasisInnerProd}
$$\langle G_{s,t},G_{u,v}\rangle=
\begin{cases}
\frac{\eta(t)}{\eta(s)} \Delta_{\omega(s)} & \text{if}\ s=u\ \text{and}\ t=v\\
0 & \text{otherwise}
\end{cases}$$
\end{lem}

\begin{rem}
	Notice that by lemma \ref{GBProd} $G_{s,s}$ is an idempotent and that by lemma \ref{GBasisInnerProd}\ $<G_{s,s},G_{s,s}>=\Delta_{\omega(s)}$.
	\end{rem}
	
\begin{prop}
\label{JMCoeff}
	In $TL_{(k,i)}$, for $j=2...k$, 
	
	$$L_{j}=\sum_{s \in T(\Lambda_{k,i})} (\lambda_{s_{j}}^{s_{j-1},i})^{2} G_{s,s}$$
	
	and
	
	$$ L_{1}=\sum_{s \in T(\Lambda_{k,i})} G_{s,s}$$
\end{prop}
\begin{proof}

Observe that projecting onto the graph basis the coefficient of $G_{s,s}$ in $L_{j}$ is

$$\frac{<L_{j},G_{s,s}>}{<G_{s,s},G_{s,s}>}=\frac{(\lambda_{s_{j}}^{s_{j-1},i})^{2} \Delta_{s_{k}} }{\Delta_{s_{k}}}$$
\end{proof}

\begin{rem}
 From the definition of the operators $L_{i}$ it is clear that they are mutually commuting. Furthermore, since each $L_{i}$ is in the span of the idempotent elements $\{G_{s,s}|s  \in T(\Lambda_{(k,i)})\}$, each of which are fixed by the map $*$, we see that $(L_{i})^{*}=L_{i}$. Lastly observe that for $s=(s_{1}, \ldots, s_{k})$
 
 $$ L_{j}\cdot G_{s,t}=(\lambda_{s_{j}}^{s_{j-1},i})^{2} G_{s,t}$$
 Thus let $c_{s}(j)=(\lambda_{s_{j}}^{s_{j-1}i})^{2}$.
 \end{rem}
 
 In what follows it will be convenient to define the collection of all eigenvalues of $L_{j}$ acting on elements of the graph basis.
 \begin{de}
 	$\mathcal{C}(j)=\{c_{s}(j) | s\in T(\Lambda_{k,i})\}$
	\end{de}
 
 \begin{prop}
 Say that for $s,t \in T(\Lambda_{k,i})$ that $L_{j}G_{s,v}=L_{j} G_{t,u}$ ( for arbitrary $v$ and $u$ with the same weight as $s$ and $t$ respectively) for $1 \leq j \leq k$. Then $s=t$.
 \end{prop}
 \begin{proof}
 From the definition of $\lambda_{p}^{qr}$ it follows directly that if $(\lambda_{s_{j}}^{s_{j-1} i})^{2}=(\lambda_{t_{j}}^{t_{j-1} i})^{2}$ and $s_{j-1}=t_{j-1}$ then $s_{j}=t_{j}$. Since $s_{1}=t_{1}=i$, considering $L_{2}G_{s,v}=L_{2} G_{t,u}$, it follows that $s_{2}=t_{2}$. Continuing on one sees that $s_{j}=t_{j}$ for all $j$. Thus $s=t$. It follows that if $s \neq t$ then $L_{j}G_{s,v} \neq L_{j} G_{t,u}$ for some j. The proposition follows.
 \end{proof}
 
The above proposition along with the previous remark show that 

\begin{thm}
$\{ L_{1}, \ldots, L_{k}\}$ are a separating family of JM-elements for $TL_{(k,i)}$
\end{thm}

 Having made this observation several results of \cite{mathas2008} follow directly. In particular, using the JM-elements, there is a basis constructed  that contains a family of primitive idempotents. In fact, the graph basis $\{G_{s,t}\}$ \textit{is} this basis. This is the focus of the following section.
 
 \section{Consequences of Cellularity}
 \label{CC}
 Consider the following element defined in \cite{mathas2008}. For $t\in T(\lambda)$ let 
 $$F_{t}= \prod_{j=1}^{k} \prod _{\stackrel{c \in \mathcal{C}(j)}{c \neq c_{t}(j)}}\frac{L_{j}-c}{c_{t}(j)-c}$$

\begin{prop}
\label{Ft}
For $t\in T(\lambda)$,
$$G_{t,t}=F_{t}$$
\end{prop}

\begin{lem}
Let $$F_{t}^{j}=\prod _{\stackrel{c \in \mathcal{C}(j)}{c \neq c_{t}(j)}}\frac{L_{j}-c}{c_{t}(j)-c}$$

Then
$$F_{t}^{j} \cdot G_{u,v} =
\begin{cases}
G_{u,v} & \text{if}\ c_{t}(j)=c_{u}(j) \\
0 & \text{otherwise}
\end{cases}
$$
\end{lem}
\begin{proof}
First suppose that $c_{u}(j) \neq c_{t}(j)$. Then $F_{t}^{j}$ contains the term
$$\frac{L_{j}\cdot G_{u,v}-c_{u}(j) G_{u,v}}{c_{t}(j)-c_{u}(j)}=\frac{c_{u}(j)-c_{u}(j)}{c_{t}(j)-c_{u}(j)}\ G_{u,v}=0$$
while if $c_{t}(j)=c_{u}(j)$ then each term is
$$\frac{L_{j}\cdot G_{u,v}-c G_{u,v}}{c_{t}(j)-c}=\frac{c_{t}(j)-c}{c_{t}(j)-c}\ G_{u,v}=G_{u,v}$$
\end{proof}

\begin{lem}
For $t\in T(\lambda)$,
$$ F_{t} \cdot G_{u,v}=
\begin{cases}
G_{u,v} & \text{if}\ t=u \\
0 & \text{otherwise}
\end{cases}
$$
\end{lem}

\begin{proof}
Notice that $F_{t}^{j} \cdot G_{u,v}=0$ for some $j$ unless $c_{t}(j)=c_{u}(j)$ for $j=1\ldots k$. However, as the JM-elements separate $TL_{(k,i)}$ we see that $F_{t} \cdot G_{u,v} \neq 0$ only if $u=t$. When $u=t$ then clearly $F_{t} \cdot G_{t,v}=G_{t,v}$ 
\end{proof}

\begin{proof}[ Proof of proposition \ref{Ft}]

Let us express $F_{t}$ in terms of the graph basis. To that end, the coefficient of $G_{v,u}$ is 

$$\frac{<F_{t},G_{v,u}>}{<G_{v,u},G_{v,u}>}$$

Notice then that 
\begin{eqnarray*}
<F_{t},G_{v,u}>&=&<F_{t} \cdot G_{u,v},Id>\\
&=&\delta_{t,u}<G_{u,v},Id>\\
&=&\delta_{t,u} \delta_{u,v}<G_{u,u},Id>\\
&=&\delta_{t,u} \delta_{u,v}<G_{u,u}\cdot G_{u,u},Id>\\
&=&\delta_{t,u} \delta_{u,v}<G_{u,u},G_{u,u}>
\end{eqnarray*}

The result follows.

\end{proof}

Mathas  uses the elements $F_{t}$ to construct a new cellular basis from the original one (\cite{mathas2008} \textit{definition 3.1}). In particular, he defines $f_{st}=F_{s}a_{st}F_{t}$, where $a_{st}$ is the original cellular basis. In the case considered here this becomes $f_{st}=G_{ss}G_{st}G_{tt}=G_{st}$. Thus, the construction of this new basis results in the original basis.  Mathas's results now apply to this graph basis, as well as several other theorems on cellular algebras with a set of separating JM-elelemts. To avoid a recapitulation of those results we state several corollaries here with reference to \cite{mathas2008}. 

\begin{cor}
The right ideals generated by the idempotent elements of the graph basis $\{ G_{s,s}TL_{k,i}=span\{G_{s,t}| \omega(s)=\omega(t)\} | s \in \Lambda_{k,i}\}$ form a complete set of pairwise non-isomorphic irreducible $TL_{k,i}$-modules.
\end{cor}
\begin{proof}
This follows directly from corollaries 3.10 and 3.11 in \cite{mathas2008}
\end{proof}
 
 \begin{cor}[Corollary to  Theorem 3.16 of \cite{mathas2008}]
 \hspace{5 pt}
 \begin{enumerate}
 \item For each $t \in T(\Lambda_{k,i})$ $G_{t,t}$ is a primitive idempotent.
 \item $\sum_{t \in T(\lambda)} G_{t,t}$ is a primitive central idempotent.
 \item $\{G_{t,t}| t \in T(\Lambda_{k,i})\}$ is a complete set of pairwise orthogonal idempotents.
 \item $\{\sum_{t \in T(\lambda)} G_{t,t}| \lambda \in T(\Lambda_{k,i})\}$ is a complete set of pairwise orthogonal idempotents.
 \end{enumerate}
 \end{cor}		
 
\begin{cor}[ Corollary 3.8 \cite{mathas2008}]
 $\{L_{1}, \ldots, L_{k}\}$ generates a maximum abelian subalgebra of $TL_{(k,i)}$.
 \end{cor}

	Having these statements together fixes the representation theory of $TL_{(k,i)}$. In particular, this allows one to view the colored Jones polynomial on the same footing as the original definition of the Jones polynomial, as the trace of an irreducible representation.



\section{Recursive tangles and Mahler measure}
\label{MMMT}

	The Mahler measure of a polynomial is defined by 
	
	\begin{de}
\label{MM}
Let $f\in\BC[z_1^{\pm1},\cdots,z_k^{\pm1}]$.
The Mahler measure of $f$ is defined to be
\begin{equation}
M(f)=e^{\int_0^1\cdots\int_0^1\log|f(e^{2\pi i\theta_1}\cdots e^{2\pi i\theta_k})|d\theta_1\cdots d\theta_k}\notag
\end{equation}
\end{de}

	When $f(z)=a_{0}z^{k} \prod_{i=1\ldots n} (z-\alpha_{i})$,  by Jensen's formula $$M(f)=|a_{0}|\prod_{i=1\ldots n} max(1,|\alpha_{i}|)$$ Clearly, the Mahler measure is multiplicative and thus $M(z^{r}f)=M(f)$. With this in mind it is natural to extend the definition of Mahler measure to rational functions of Laurent polynomials.
	
	Several results connect the study of the hyperbolic volume and the Mahler measure of the Jones and colored Jones polynomials  of knots. This connection seems to be mediated by regions of twisting in a knot. In particular, it is a well known result that the twist number of an alternating knot provides upper and lower bounds on the hyperbolic volume of the knot compliment (\cite{lackenby2004} and \cite{brittenham}), i.e., the set of volumes of alternating knots with bounded twist number is bounded. Similarly, if $\mathcal{D}$ is a set of oriented link diagrams with bounded twist number then the set of Mahler measures of the associated Jones polynomials is bounded (\cite{ssilver2006}). A standard result of Thurston says that if $L_{m}$ is obtained from hyperbolic knot $L$ by adding $m$ full twists on $n$-strands then the hyperbolic volume of $L_{m}$ converges as $m \rightarrow \infty$. Again this result is translated to the context of Mahler measure in \cite{CK}; the Mahler measure of the Jones and colored Jones polynomials of a link converges when adding full twists on $n$-strands. 
	
	Paramount to these results on Mahler measure are two important facts. First is that the Mahler measure of a polynomial is bounded above by its length (the sum of the absolute value of its coefficients). Second is theorem \ref{converge}. 
	
	\begin{thm}\cite{La}
\label{converge}
For every $f\in\BC[z_1^{\pm1},\cdots,z_k^{\pm1}]$,
\begin{equation}
M(f)=\lim_{v(x_1,\cdots,x_k)\rightarrow\infty}M(f(z^{x_1},\cdots,z^{x_k})).
\notag
\end{equation}
where
\begin{equation}
v(x_1,\cdots,x_k)=\min\{\max|y_i|\ s.t.\ (y_1,\cdots,y_k)\in\BZ^k, (x_1,\cdots,x_k)\cdot (y_1,\cdots,y_k)=0\}.
\notag
\end{equation}
with the following special case
$$M(f)=lim_{d \rightarrow \infty} M(f(z,z^{d},\ldots, z^{d^{s-1}}))$$
\end{thm}
	
	In both \cite{CK} (section 2 equation (3)) and \cite{ssilver2006} (Lemma 3.1) showing that the Jones polynomial took a particular form was a major step towards their respective results. In fact, the form that the Jones polynomial took was that of a recursive polynomial, defined as follows.	
	\begin{de}
	\label{recursive}
		Call a family of polynomials $\{p_{k}(z)\}$  recursive if 
		$$p_{k}(z)=\sum_{i=1}^{n} \alpha_{i}(z)^{k} q_{i}(z)$$
		\end{de}

	\begin{cor}
		If $$p_{k}(z)=\sum_{i=1}^{n} \alpha_{i}(z)^{k} q_{i}(z)$$ and $\alpha_{i}(z)$ is a monomial in $z$, then the Mahler measure of $p_{k}$, $M(p_{k}(z))$, converges as $k$ goes to infinity.
		\label{mahlerrecform}
		\end{cor}
 	Notice that in the case of a recursive polynomial of the form in definition \ref{recursive}, if for each $i$, $\alpha_{i}(z)$ is a power of $z$ then the Mahler measure of the family of polynomials converges to the Mahler measure of a polynomial in $2$ variables. Furthermore, the second author in \cite{todd12} built several families of links whose Jones polynomials are recursive and showed that their Mahler measures are divergent.

 	With these observations in mind, elements of $TL_{(k,i)}$ that behave recursively are of interest.

	\begin{de}
		Let a $ki$-tangle R be called recursive if as an element of $TL_{(k,i)}$ 
		$$R=\sum_{s \in T(\lambda), \lambda \in \Lambda_{k,i} } \alpha_{s} G_{s,s}$$
		\end{de}
		
		Suppose that R is a recursive $ki$-tangle and $T$ is any $ki$-tangle. Computing the Jones and colored Jones polynomials of the links formed by adding an arbitrary number of copies of $R$ to $T$ is similar to adding twists on some number of strands in a link. That is, it would be interesting to compute $<R^{n},T>$ in $TL_{(k,i)}$.
		
		\begin{prop}
		Say that $R=\sum_{s \in T(\Lambda_{k,i}) } \alpha_{s} G_{s,s}$ is a recursive element  and $T=\sum_{u,v} \beta_{u,v} G_{u,v}$ in $TL_{(k,i)}$. Then 
		
		$$<R^{n},T>=\sum_{u} \alpha_{u}^{n} \beta_{u,u} \Delta_{\omega(u)}$$
		\label{recform}
		\end{prop}

\begin{proof}

	First notice that $R$ is a linear combination of orthogonal idempotents and thus $R^{n}= \sum_{s} \alpha_{s}^{n} G_{s,s}$. By lemma \ref{GBProd} each non-zero term of the inner product looks like 
	
	$$ \alpha_{s}^{n} \beta_{s,t} <G_{s,s}, G_{s,t}>$$
	
	Then applying lemma \ref{GBasisInnerProd} the result follows.
	
\end{proof}

\begin{thm}
	If $R$ is any product of the JM-elements in $TL_{(k,i)}$ then $R=\sum_{s \in T(\Lambda_{k,i}) } \alpha_{s} G_{s,s}$ and each $\alpha_{s}$ is a power of A.  Furthermore if $T=\sum_{u,v} \beta_{u,v} G_{u,v}$ in $TL_{(k,i)}$ then
	$M(<R^{n},T>)$ converges as $n$ goes to infinity.
	\end{thm}
	
\begin{proof}
	Since the family of JM-elements are mutually commuting (and $L_{1}=Id$) we may write 
	$$R=L_{2}^{p_{2}}L_{3}^{p_{3}}\ldots L_{k}^{p_{k}}$$
	
	Using proposition \ref{JMCoeff} we find that 
	
	$$R=\sum_{s} (\prod_{j=2}^{k} (\lambda_{s_{j}}^{s_{kj1},i})^{2*p_{j}}) G_{s,s}$$
	Furthermore, $\lambda_{s_{j}}^{s_{j-1},i}$ is a positive power of A for any $(s_{j},s_{j-1},i)$ admissible \cite{kauffman94}. The rest of the theorem follows directly from proposition \ref{recform} and proposition \ref{mahlerrecform}
	
	\end{proof}

	We end this section by noting that the above theorem provides a single proof for the fact that the Mahler measure of the Jones and colored Jones polynomial converges under twisting on multiple strands. Consider the full twist on $n$ strands as an element of $T_{(n,k)}$ as depicted in figure \ref{ft1}.  From this figure it is clear that in fact the full twist on $n$ $k$-colored strands is the product of the JM-elements, $F_{n}=L_{2}\ldots L_{n}$. Thus from the above theorem and proposition \ref{mahlerrecform} we arrive at the results given by Champanerkar and Kofman in 	\cite{CK}, as stated in the following corollary.
\begin{figure}[ht]
\begin{center}
\includegraphics{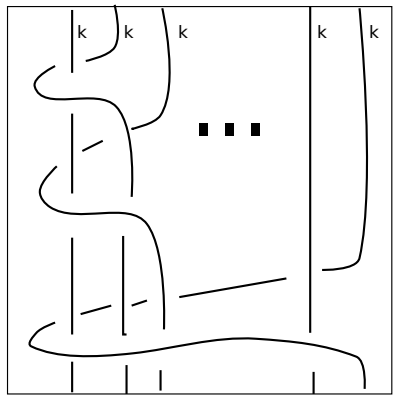}
\end{center}
\caption{Full twist on $n$ strands}
\label{ft1}
\end{figure}

\begin{cor}
The Mahler measure of the Jones and $n^{th}$-colored Jones polynomial converge under the consecutive addition of full twists on $n$ strands.
\end{cor}

\section*{Acknowledgements}
XC is partially supported by research assistantship funded by NSF-DMS-0905736.

\bibliographystyle{plain}
\bibliography{Bib}










\end{document}